\documentclass[11pt]{amsart}
\usepackage{amssymb,latexsym}
\usepackage{amsthm}
\usepackage{amsmath}
\usepackage{amsfonts}
\usepackage{graphicx}
\usepackage[all]{xy}
\usepackage{fancyhdr}
\usepackage[top=1in,bottom=1in,left=1.25in,right=1.25in]{geometry}

\newtheorem{theorem}{Theorem}[section]
\newtheorem{definition}[theorem]{Definition}
\newtheorem{remark}[theorem]{Remark}
\newtheorem{proposition}[theorem]{Proposition}

\def\Q{\mathbb{Q}}
\def\F{\mathbb{F}}
\def\R{\mathbb{R}}
\def\Z{\mathbb{Z}}

\def\Cm{\mathbb{C}}
\def\B{\mathcal{B}}

\def\D{\mathcal{D}}
\def\I{\mathcal{I}}

\def\M{\mathcal{M}}

\def\ra{\rightarrow}
\def\ov{\overline}
\def\st{\stackrel}
\def\wh{\widehat}

\usepackage[pdftex]{hyperref}
\usepackage{fancyhdr}

\begin{document}

\title{On the Lefschetz trace formula for Lubin-Tate spaces}
\author{Xu SHEN}
\date{}
\address{D\'epartement de Math\'ematiques d'Orsay\\
Universit\'e Paris-Sud 11\\
Orsay, 91405, France} \email{xu.shen@math.u-psud.fr}

\begin{abstract}
We reprove the Lefschetz trace formula for Lubin-Tate spaces, based on the locally finite cell decompositions of these spaces obtained in \cite{F3},\cite{F4}, and Mieda's theorem of Lefschetz trace formula for certain open adic spaces (\cite{Mi1} theorem 3.13). This proof is rather different from those in \cite{St} (theorem 3.3.1) and \cite{Mi1} (example 4.21), and is quite hopeful to generalized to some other Rapoport-Zink spaces as soon as there exist some suitable cell decompositions. For example, we proved a Lefschetz trace formula for some unitary group Rapoport-Zink spaces in \cite{Sh} by using similar ideas here.
\end{abstract}
\maketitle

\section{Introduction}
Let $p$ be a prime number, $F$ be a finite extension of $\Q_p$, $O$ be the ring of integers of $F$, and $\pi\in O$ be a uniformizer in $O$. We denote $\wh{F}^{nr}$ as the completion of the maximal unramified extension of $F$, and $\wh{O}^{nr}$ its ring of integers. For any integer $n\geq 1$, we consider the general linear group $GL_n$ as well as its inner form $D^\times$ over $F$, where $D$ is the central division algebra over $F$ with invariant $\frac{1}{n}$ and $D^\times$ is the reductive group defined by inverse elements in $D$. Recall a formal $O$-module is a $p$-divisible group with an $O$-action over a base over $O$, such that the induced action on its Lie algebra is the canonical action of $O$. We consider the formal Lubin-Tate space $\wh{\M}=\coprod_{i\in\Z}\wh{\M}^i$ over $\wh{O}^{nr}$: for any scheme $S\in \textrm{Nilp}\wh{O}^{nr}$,
$\wh{\M}(S)=\{(H,\rho)\}/\simeq$, where
\begin{itemize}
\item $H$ is a formal $O$-module over $S$,
\item $\rho: \mathbb{H}_{\ov{S}}\ra H_{\ov{S}}$ is a quasi-isogeny.
\end{itemize}
Here $\textrm{Nil}\wh{O}^{nr}$ is the category of schemes over $\wh{O}^{nr}$ on which $\pi$ is locally nilpotent, $\mathbb{H}$ is the unique (up to isomorphism) formal $O$-module over $\ov{\F}_p$ with $O$-height $n$, and $\ov{S}$ is the closed subscheme defined by $\pi$ of $S\in \textrm{Nilp}\wh{O}^{nr}$. For $i\in\Z$, $\wh{\M}^i$ is the open and closed subspace of $\wh{\M}$ such that the quasi-isogenies $\rho$ have $O$-height $i$. There is a natural (left) action of $D^\times$ on $\wh{\M}$ by $\forall b\in D^\times, b:\wh{\M}\ra\wh{\M},\; (H,\rho)\mapsto(H,\rho\circ b^{-1})$. This action induces non-canonical isomorphisms \[\wh{\M}^i\simeq \wh{\M}^0,\]while one knows that there is a non-canonical isomorphism\[\wh{\M}^0\simeq Spf(\wh{O}^{nr}[[x_1,\dots,x_{n-1}]]).\]

Let $\M=\wh{\M}^{an}=\coprod_{i\in\Z}\M^i$ be the Berkovich analytic fiber of $\wh{\M}$. By trivializing the local system over $\M$ defined by the Tate module of $p$-divisible group, we have the Lubin-Tate tower $(\M_K)_{K\subset GL_n(O)}$ over $\wh{F}^{nr}$, and the group $GL_n(F)$ acts (on right) on this tower through Hecke correspondences. When $K=K_m:=ker(GL_n(O)\ra GL_n(O/\pi^mO))$ for some integer $m\geq 0$, there is a regular model $\wh{\M}_m$ of $\M_{K_m}$ by introducing Drinfeld structures on $O$-modules. We will not use these models and we will work always on the Berkovich spaces $\M_K$. Note there are natural actions of $D^\times$ on each $\M_K$, which commute with the Hecke action.

Fix a prime $l\neq p$, let $\ov{\Q}_l$ (resp. $\ov{\Q}_p$) be a fixed algebraic closure of $\Q_l$ (resp. $\Q_p$), and $\Cm_p$ be the completion of $\ov{\Q}_p$. For each $i\geq0$, we consider the cohomology with compact support
\[H^i_c(\M_K\times\Cm_p,\ov{\Q}_l)=\varinjlim_{U}\varprojlim_{n}H^i_c(U\times\Cm_p,\Z/l^n\Z)\otimes\ov{\Q}_l,\]
where the injective limit is taken over all locally compact open subsets $U\subset\M_K$, see \cite{F1} section 4 and \cite{Hu2}. We have
\[H^i_c(\M_K\times\Cm_p,\ov{\Q}_l)=\bigoplus_{j\in\Z}H^i_c(\M_K^j\times\Cm_p,\ov{\Q}_l),\]
where \[dim_{\ov{\Q}_l}H^i_c(\M_K^j\times\Cm_p,\ov{\Q}_l)<\infty\]by theorem 3.3 in \cite{Hu2}. In fact we have also the usual $l$-adic cohomology groups $H^i(\M_K^j\times\Cm_p,\ov{\Q}_l)$ which are Poincar\'e dual to those $H^i_c(\M_K^j\times\Cm_p,\ov{\Q}_l)$, and (cf. \cite{St} lemma 2.5.1)
\[H^i_c(\M_K^j\times\Cm_p,\ov{\Q}_l)\neq 0\Leftrightarrow n-1\leq i\leq 2(n-1),\]
\[H^i(\M_K^j\times\Cm_p,\ov{\Q}_l)\neq 0\Leftrightarrow 0\leq i\leq n-1.\]
The groups
\[\varinjlim_{K}H^i_c(\M_K\times\Cm_p,\ov{\Q}_l)\] are natural smooth representations of $GL_n(F)\times D^\times\times W_F$ ($W_F$ is the Weil group of $F$), and the local Langlands and Jacquet-Langlands correspondences between the three groups were proved realized in these groups, see \cite{C} and \cite{HT}.

In \cite{St} Strauch had proven a Lefschetz trace formula for regular elliptic elements action on the Lubin-Tate spaces. More precisely, we consider
\[H^\ast_c(\M_K^j\times\Cm_p,\ov{\Q}_l)=\sum_{i}(-1)^iH^i_c(\M_K^j\times\Cm_p,\ov{\Q}_l).\]
Let $\gamma=(g,b)\in GL_n(F)\times D^\times$ such that $g,b$ are both regular elliptic elements, $gKg^{-1}=K$, $v_p(detg)+v_p(Nrdb)=0$ ($Nrd: D^\times\ra F^\times$ is the reduced norm and $v_p$ is the valuation on $F$), then we have an automorphism \[\gamma:\M_K^j\ra\M_K^j,\]which induces morphism on cohomology groups
\[\gamma: H^i_c(\M_K^j\times\Cm_p,\ov{\Q}_l)\ra H^i_c(\M_K^j\times\Cm_p,\ov{\Q}_l).\]We define
\[Tr(\gamma|H^\ast_c(\M_K^j\times\Cm_p,\ov{\Q}_l)):=\sum_{i}(-1)^iTr(\gamma|H^i_c(\M_K^j\times\Cm_p,\ov{\Q}_l)).\]
Strauch proved the following trace formula.
\begin{theorem}[\cite{St}, Theorem 3.3.1]  Under the above assumptions and notations, we have
\[Tr(\gamma|H^\ast_c(\M_K^j\times\Cm_p,\ov{\Q}_l))=\#\textrm{Fix}(\gamma|\M_K^j\times\Cm_p).\]
\end{theorem}
By applying the $p$-adic period mapping \[\M\ra\mathbf{P}^{n-1,an},\] Strauch obtained a nice fixed points number formula for the quotient space $\M_K/\pi^\Z$ (theorem 2.6.8 in loc. cit.)
\[\#\textrm{Fix}(\gamma|(\M_K/\pi^\Z)(\Cm_p))=n\#\{h\in GL_n(F)/\pi^\Z K|h^{-1}g_bh=g^{-1}\},\]
which can be rewritten as some suitable orbit integral, see \cite{Mi2} proposition 3.3. This Lefschetz trace formula enable Strauch to prove the Jacquet-Langlands correspondence between smooth representations of $GL_n(F)$ and $D^\times$ is realized the cohomology of the tower $(\M_K)_{K}$, not involving with Shimura varieties as in \cite{HT}, see section 4 of \cite{St}.

There are two main ingredients in Strauch's proof of the above theorem. The first is some careful approximation theorems of Artin in this special (affine) case, and the second is Fujiwara's theorem of specialization of local terms (\cite{Fu} proposition 1.7.1). In general case one has no sufficient approximation theorems, thus his method can be hardly generalized. In \cite{Mi1} Mieda proved a general Lefschetz trace formula for some open smooth adic spaces by totally working in rigid analytic geometry, and verified his conditions in the special case of Lubin-Tate spaces hold, thus he can reprove the above Lefschetz trace formula. Both Strauch and Mieda worked in the category of adic spaces, and study the action of $\gamma$ on the boundary strata (outside the corresponding Berkovich space) of the analytic generic fiber of $\wh{\M}_m$. Their boundary stratas are at last linked to the theory of generalized canonical subgroups (cf. \cite{F2} section 7), thus their approachs can hardly be generalized.

In this note we work mainly with Berkovich spaces. We will consider Fargues's locally finite cell decomposition of Lubin-Tate spaces, cf. \cite{F3} chapter 1 and \cite{F4}. By studying the action of $\gamma$ on these cells, we verify the conditions in Mieda's theorem of Lefschetz trace formula hold, by the dictionary between the equivalent categories of Hausdorff strictly Berkovich $k$-analytic spaces and adic spaces which are taut and locally of finite type over $Spa(k,k^0)$. ($k$ is a complete non-archimedean field and $k^0$ is its ring of integers.) Thus we can reprove the above theorem, by a different method. The advantage of our method is that, once we know there exists a locally finite cell decomposition, with the fundamental domain compact, then by studying the action on the cells we will easily verify Mieda's theorem applies. For example, we can treat the case of some unitary group Rapoport-Zink spaces in \cite{Sh}, and we will also treat the case of the basic Rapoport-Zink spaces for $GSp4$.

In next section we review the locally finite cell decomposition of $\M_K$, and in section 3 we study the action of $\gamma$ on the cells and verify Mieda's theorem applies.\\
 \\
\textbf{Acknowledgments.} I would like to thank Laurent Fargues sincerely for sharing the idea of cell decomposition of Rapoport-Zink spaces, and the suggestion that this will suffice to prove the Lefschetz trace formula. I would like to thank also Bertrand R\'emy sincerely for replying my question on buildings, and pointing me the article \cite{R}. Yoichi Mieda had proposed me some useful questions after the first version of this article, so special thanks go to him as well.

\section{The locally finite cell decomposition of Lubin-Tate spaces}

 In \cite{F3} and \cite{F4} Fargues found some locally finite cell decompositions of $\M_K$. The parameter set of cells in \cite{F3} is the set of vertices of some Bruhat-Tits building, and these cells for $K$ varies form in fact a cell decomposition of the tower $(\M_K)_K$ but not for a fixed space $\M_K$. Therefore we will mainly follow the construction in \cite{F4}, where the parameter set is essentially some set of Hecke correspondences. To consider the group actions on these cells, we will relate the parameter set with a Bruhat-Tits building by borrowing some ideas from \cite{F3}.

 First consider the case without level structures. Fix a uniformizer $\Pi\in D^\times$, then \[\Pi^{-1}:\M^i\st{\sim}{\longrightarrow}\M^{i+1}.\]Let $\M^{ss}$ be the semi-stable locus in $\M$, i.e. the locus where the associated $p$-divisible groups are semi-stable in the sense of \cite{F4} definition 4, which is a closed analytic domain in $\M$. Let $\D=\M^{ss,0}:=\M^{ss}\bigcap\M^0$, then $\M^{ss}=\coprod_{i\in\Z}\Pi^{-i}\D$ and $\D$ is the compact fundamental domain of Gross-Hopkins, see \cite{F3} 1.5. The main results of \cite{F4} for our special case say that we have a locally finite covering \[\M=\bigcup_{\begin{subarray}{c}T\in GL_n(O)\setminus GL_n(F)/GL_n(O)\\i=0,\dots,n-1\end{subarray}}T.\Pi^{-i}\D,\]
where $T.A$ is the image under the Hecke correspondence $T$ for a subset $A$, which is an analytic domain if $A$ is.
In the following we shall actually work with one component $\M^0$, so we consider its induced cell decomposition
\[\M^0=\bigcup_{\begin{subarray}{c}T\in GL_n(O)\setminus GL_n(F)/GL_n(O)\\i=0,\dots,n-1\end{subarray}}((T.\Pi^{-i}\D)\bigcap\M^0).\] For $T\in GL_n(O)\setminus GL_n(F)/GL_n(O), i=0,\dots,n-1$,
\[(T.\Pi^{-i}\D)\bigcap\M^0\neq \emptyset\Leftrightarrow -v_p(detT)+i=0,\]in which case \[T.\Pi^{-i}\D\subset\M^0.\]
 Here ($v_p: F^\times\ra\Z$ is the valuation of $F^\times$ such that $v_p(\pi)=1$) the composition $v_p\circ det: GL_n(F)\ra\Z$ factors through $GL_n(O)\setminus GL_n(F)/GL_n(O)\ra\Z$. Thus we have
\[ \M^0=\bigcup_{\begin{subarray}{c}T\in GL_n(O)\setminus GL_n(F)/GL_n(O)\\i=0,\dots,n-1\\-v_p(detT)+i=0\end{subarray}}T.\Pi^{-i}\D.\]

Let $K\subset GL_n(O)$ be an open compact subgroup, $\pi_K:\M_K\ra\M$ be the natural projection. We set
\[\D_K=\pi_K^{-1}(\D),\]which is a compact analytic domain in $\M_K^0$. Since the group $GL_n(O)$ acts trivially on $\M$, any element in this group will stabilize $\D_K$. Thus for two Hecke correspondences $T_1,T_2\in K\setminus GL_n(F)/K$ having the same image under the projection $K\setminus GL_n(F)/K\ra GL_n(O)\setminus GL_n(F)/K$, we have
$T_1\Pi^{-i}\D_K=T_2\Pi^{-i}\D_K$ ($\Pi^{-i}\D_K=\pi_K^{-1}(\Pi^{-i}\D)$ since $\pi_K$ is $D^\times$-equivariant). Therefore, we have the following locally finite cell decomposition in level $K$
\[\M_K=\bigcup_{\begin{subarray}{c}T\in GL_n(O)\setminus GL_n(F)/K\\i=0,\dots,n-1\end{subarray}}T.\Pi^{-i}\D_K.\]
We will denote the cells $T.\Pi^{-i}\D_K$ by \[\D_{T,i,K},\] which are compact analytic domains. For any $T\in GL_n(O)\setminus GL_n(F)/K, i\in\Z$, we denote also $\D_{T,i,K}=T.\Pi^{-i}\D_{K}$.
Since the (right) action of $F^\times$ on $\M_K$ through $F^\times\ra GL_n(F), z\mapsto z$ is the same as the (left) action of it on $\M_K$ through
$F^\times\ra D^\times,z\mapsto z$, $(z,z^{-1})\in GL_n(F)\times D^\times$ acts trivially on $\M_K$. We have
\[\D_{T,i,K}=\D_{Tz,i+nv_p(z),K}.\]
If $g\in GL_n(F)$ is an element such that $gKg^{-1}=K$, and $b\in D^\times$ is an arbitrary element, set \[\gamma:=(g,b).\] Then automorphism $\gamma:\M_K\ra\M_K$ naturally induces an action of $\gamma$ on the set of cells of $\M_K$:
\[\gamma(\D_{T,i,K})=\D_{Tg, i-v_p(Nrdb), K}.\]
Here $Nrd: D^\times\ra F^\times$ is the reduced norm.

For the component $\M_K^0$, for $T\in GL_n(O)\setminus GL_n(F)/K, i=0,\dots,n-1$,
\[(T.\Pi^{-i}\D_K)\bigcap\M_K^0\neq \emptyset\Leftrightarrow -v_p(detT)+i=0,\]in which case
\[T.\Pi^{-i}\D_K\subset\M_K^0.\]Thus
  we have a locally finite cell decomposition
\[\begin{split}\M_K^0&=\bigcup_{\begin{subarray}{c}T\in GL_n(O)\setminus GL_n(F)/K\\i=0,\dots,n-1\end{subarray}}((T.\Pi^{-i}\D_K)\bigcap\M_K^0)\\
&=\bigcup_{\begin{subarray}{c}T\in GL_n(O)\setminus GL_n(F)/K\\i=0,\dots,n-1\\-v_p(detT)+i=0 \end{subarray}}\D_{T,i,K}.\end{split}
\]In fact for any $i\in\Z, T\in GL_n(O)\setminus GL_n(F)/K$ such that $-v_p(detT)+i=0$, we have $\D_{T,i,K}\subset\M_K^0$ with the convention above. However one can always by multiplying with some $z\in F^\times$ to reduce to the cases $0\leq i\leq n-1$. Let $\gamma=(g,b)\in GL_n(F)\times D^\times$ be such that $gKg^{-1}=K, v_p(detg)+v_p(Nrdb)=0$, then the action of $\gamma$ on $\M_K$ induces $\gamma:\M_K^0\ra\M_K^0$. In this case $\gamma$ acts on the set of cells of $\M_K^0$ as in the same way as above.

To understand better the parameter set of cells of $\M_K^0$, we look at some ideas from \cite{F3}. Consider the embedding $\mathbb{G}_m\ra GL_n\times D^\times, z\mapsto (z,z^{-1})$ of algebraic groups over $F$. Let $\B(GL_n\times D^\times,F)$ be the (extended) Bruhat-Tits building of $GL_n\times D^\times$ over $F$, and $\B=\B(GL_n\times D^\times,F)/F^\times$ be its quotient by the action of $F^\times$ through the above embedding. The set $\B^0$ of vertices of $\B$, which we define by the quotient of the vertices of $\B(GL_n\times D^\times,F)$, can be described as the set of equivalent classes \[\{(\Lambda,M)|\Lambda\subset F^n\; \textrm{is an $O$-lattice}\;,M\subset D\;\textrm{is an $O_D^\times$-lattice}\}/\sim,\]where
\[(\Lambda_1,M_1)\sim (\Lambda_2,M_2)\Leftrightarrow \exists i\in\Z, \Lambda_2=\Lambda_1\pi^i, M_2=\pi^{-i}M_1,\]
see \cite{F3} 1.5. We can understand $\B$ in the following way. The (extended) Bruhat-Tits building of $GL_n$ over $F$ is the product $\B(PGL_n,F)\times\R$ of the Bruhat-Tits building of $PGL_n$ with $\R$, while the (extended) Bruhat-Tits building of $D^\times$ over $F$ is
$\B(D^\times,F)\simeq \R$. By construction \[\B=\B(GL_n\times D^\times,F)/F^\times\simeq (\B(PGL_n,F)\times\R\times\R)/\sim,\] where $(x,s,t)\sim (x',s',t')\Leftrightarrow x=x',s-s'=t'-t=nr$ for some $r\in\Z$. Thus any point $[x,s,t]$ in $\B$ can be written uniquely in the form $[x,s',t']$ for $x\in \B(PGL_n,F),s'\in\R, t'\in [0,n)$. The elements $(g,b)\in GL_n(F)\times D^\times$ act on $\B$ by $\forall [x,s,t]\in\B$, \[(g,b)[x,s,t]=[g^{-1}x,s+v_p(detg),t+v_p(Nrdb)].\]If we consider the right action of $GL_n(F)$ on $\B(PGL_n,F)$ by $xg:=g^{-1}x$, then we can also write $(g,b)[x,s,t]=[xg,s+v_p(detg),t+v_p(Nrdb)]$.

  On the other hand, consider the action of $F^\times$ on $GL_n(O)\setminus GL_n(F)\times D^\times/O_D^\times$ by $z(GL_n(O)g,dO_D^\times)=(GL_n(O)gz,zdO_D^\times),\;\forall z\in F^\times$, then the quotient set
\[(GL_n(O)\setminus GL_n(F)\times D^\times/O_D^\times)/F^\times\]is naturally identified with the set $\B^0$ after fixing the vertex $[O^n,O_D]\in \B^0$. For an element $[GL_n(O)g,dO_D^\times]$, the associated point in $\B^0$ can be written as $[GL_n(O)F^\times g,v_p(detg),v_p(detd)]$. Here $GL_n(O)F^\times g \in \B(PGL_n,F)$ by fixing the homothety class of $O^n$. Now let $K\subset GL_n(O)$ be an open compact subgroup, then the set
\[\I_K:=(GL_n(O)\setminus GL_n(F)/K\times D^\times/O_D^\times)/F^\times\]can be identified with the image $\B^0/K$ of $\B^0$ in the quotient space $\B/K$. If $\gamma=(g,b)\in GL_n(F)\times D^\times$ such that $gKg^{-1}=K$, then $\gamma$ acts on the set $\I_K$ by $[T,d]\mapsto [Tg,bd]$.
 There are two natural projection maps $\I_K\ra (GL_n(O)\setminus GL_n(F)/K)/F^\times$ and $\I_K\ra (D^\times/O_D^\times)/F^\times\simeq \Z/n\Z$. There is as well as a map
 \[\begin{split}GL_n(O)\setminus GL_n(F)/K&\times D^\times/O_D^\times\longrightarrow\Z\\
 &(T,d)\mapsto -v_p(detT)-v_p(Nrdd).\end{split}\]Let $(GL_n(O)\setminus GL_n(F)/K\times D^\times/O_D^\times)^0$ be the inverse image of $0$ under this map. Since the action of $F^\times$ does not change the values of the above map, it factors through $\I_K\ra\Z$.
 In fact there is a well defined continuous map
 \[\begin{split}\varphi: \B&\longrightarrow\R\\ [x,s,t]&\mapsto -s-t,\end{split}\] with each fiber stable under the action of $K$. The above map is induced by $\varphi$.
 For the $\gamma$ as above with further condition that $v_p(detg)+v_p(Nrdb)=0$, it stabilizes the subset
 \[\I_K^0:=(GL_n(O)\setminus GL_n(F)/K\times D^\times/O_D^\times)^0/F^\times\]for the above action. For the map $\varphi$ above, we see that $\I_K^0$ is identified with the quotient set $\varphi^{-1}(0)^0/K$ of vertices in $\varphi^{-1}(0)$.

 For any element $[T,d]\in \I_K$, the cell $[T,d]\D_K$ is well defined, which is what we denoted by $\D_{T,-v_p(Nrdd),K}$ above. As before we denote $[T,d]\D_K$ as \[\D_{[T,d],K}.\] Then we can rewrite the cell decompositions as
\[\begin{split}
&\M_K=\bigcup_{[T,d]\in \I_K}\D_{[T,d],K},\\
&\M_K^0=\bigcup_{[T,d]\in \I_K^0}\D_{[T,d],K}.
\end{split}\]
For $\gamma=(g,b)\in GL_n(F)\times D^\times$ as above, it acts on the cells in the way compatible with its action on $\I_K$:
\[\gamma(\D_{[T,d],K})=\D_{[Tg,bd],K}.\]

Recall there is a metric $d(\cdot,\cdot)$ on $\B$, so that $(GL_n(F)\times D^\times)/F^\times$ acts on it by isometries. If $d'(\cdot,\cdot)$ is the metric on $\B(PGL_n,F)$, then for two points $[x,s,t], [x',s',t']$ with $x,x'\in \B(PGL_n,F), s,s'\in\R, t,t'\in [0,n)$ we have \[d([x,s,t],[x',s',t'])=\sqrt{d'(x,x')^2+(s-s')^2+(t-t')^2}.\] The group $K$ acts on $\B$ through the natural morphisms $K\ra GL_n(F)\times D^\times\ra (GL_n(F)\times D^\times)/F^\times$. There is an induced metric $\ov{d}(\cdot,\cdot)$ on the quotient space $\B/K$:
\[\ov{d}(xK,yK):=\inf_{k,k'\in K}d(xk,yk')=\inf_{k\in K}d(xk,y)=\inf_{k\in K}d(x,yk),\;\forall\; xK,yK\in\B/K,\]
the last two equalities come from $d(xk,yk')=d(xk(k')^{-1},y)=d(x,yk'k^{-1})$. Since $K$ is compact, one checks it easily that this is indeed a metric on $\B/K$. With this metric, $\I_K,\I_K^0$ are both infinity discrete subset of $\B/K$, and any closed ball in $\B/K$ contains only finitely many elements of $\I_K$ and $\I_K^0$. We will directly work with the induced metric space
\[\I_K=\B^0/K.\]For $\gamma=(g,b)\in GL_n(F)\times D^\times$ with
$gKg^{-1}=K$, one can check by definition of $\ov{d}$ that the above action of $\gamma$ on $\I_K$ is isometric:
\[\ov{d}(\gamma x,\gamma x)=\ov{d}(x,x),\;\forall x\in \I_K.\]
Note that for $[T_1,d_1],[T_2,d_2]\in \I_K$, $\D_{[T_1,d_1],K}\bigcap\D_{[T_2,d_2],K}\neq\emptyset$ implies that
$v_p(detT_1)+v_p(Nrdd_1)=v_p(detT_2)+v_p(Nrdd_2)$. If we write $[T_1,d_1]=[x_1K,s_1,t_1], [T_2,d_2]=[x_2K,s_2,t_2]$ with $x_1,x_2\in \B(PGL_n,F), s_1,s_2\in\Z\subset\R, t_1,t_2\in[0,n)\bigcap\Z$ (i.e. $\exists r_1,r_2\in\Z, s.t.\; v_p(detT_i)=s_i+nr_i,v_p(Nrdd_i)=t_i-nr_i, i=1,2$), then $s_1+t_1=s_2+t_2, s_1-s_2=t_2-t_1\in [1-n,n-1]$, the distance \[\ov{d}([T_1,d_1],[T_2,d_2])=\inf_{k\in K}\sqrt{d'(x_1,x_2k)^2+2(s_1-s_2)^2}\]just depends on $\ov{d'}(x_1K,x_2K)$ for the induced metric $\ov{d'}$ on $\B(PGL_n,F)$ defined in the same way as $\ov{d}$.
By the construction of the locally finite sell decomposition of $\M_K$, we have the following proposition.
\begin{proposition}
There exists a constant $c>0$, such that for any $[T_1,d_1],[T_2,d_2]\in \I_K$ with $\ov{d}([T_1,d_1],[T_2,d_2])>c$, we have
\[\D_{[T_1,d_1],K}\bigcap\D_{[T_2,d_2],K}=\emptyset.\]
\end{proposition}
\begin{proof}
We need to prove that, there exists a constant $c>0$, such that for any $[T,d]\in\I_K$, and any $[T',d']\in\{[T',d']\in\I_K|\D_{[T',d'],K}\bigcap \D_{[T,d],K}\neq\emptyset\}$, we have $\ov{d}([T,d],[T',d'])\leq c$.
This just comes from the construction of the locally finite cell decomposition of $\M_K$, and the definition of $\ov{d}$. We just indicate some key points. First, for any fixed choice of fundamental domain $V_K$ in $\B$ for the action of $K$, by definition $\forall x,y\in V_K, d(x,y)\geq \ov{d}(xK,yK)$. Next, by the proof of proposition 24 of \cite{F4}, and the Cartan decomposition $GL_n(O)\setminus GL_n(F)/GL_n(O)\simeq \Z^n_+=\{(a_1,\dots,a_n)\in\Z^n|a_1\geq\cdots\geq a_n\}$, for any fixed the Hecke correspondence $T\in GL_n(O)\setminus GL_n(F)/GL_n(O), i\in\Z$, the finite set \[A_{[T,i]}:=\{[T',j]\in (GL_n(O)\setminus GL_n(F)/GL_n(O)\times \Z)/F^\times| T.\Pi^{-i}\D\bigcap T'.\Pi^{-j}\D\neq \emptyset\}\] is such that $\forall [T',j]\in A_{[T,i]}$ we have $-v_p(detT')+j=-v_p(detT)+i$; and if $T$ corresponds to the point $(a_1,\dots,a_n)\in\Z^n_+$, then for $j\in \Z/n\Z$ fixed, the set $T'\in GL_n(O)\setminus GL_n(F)/GL_n(O)$ with $[T',j]\in A_{[T,i]}$, correspond to the points $(a_1',\dots,a'_n)\in\Z^n_+$ such that $\sum_{k=1}^na_k'=\sum_{k=1}^na_k -i+j\;(mod\; n\Z), |a_k-a_k'|\leq C $ for all $k=1,\dots,n$, and $C>0$ is a constant doesn't depend on $[T,i]$. From these two points one can easily deduce the proposition for $K=GL_n(O)$, and the general case will be obtained as soon as the case $K=GL_n(O)$ holds.
\end{proof}

We remark that, in \cite{F2} Fargues defined an $O_{D^\times}$-invariant continuous map of topological spaces
\[\M^0\longrightarrow \B(PGL_n,F)/GL_n(O),\]
and identified the image of $\D$ under this map. However, this map depends quite on our special case, and in general there is no such a map from Rapoport-Zink spaces to Bruhat-Tits buildings. For any open compact subgroup $K\subset GL_n(O)$, there is also a continuous map $\M^0\ra \B(PGL_n,F)/K$, and we have a commutative diagram of continuous maps between topological spaces
\[\xymatrix{\M_K^0\ar[r]\ar[d]&\B(PGL_n,F)/K\ar[d]\\
\M^0\ar[r]&\B(PGL_n,F)/GL_n(O).}\]These maps are Hecke equivariant, thus compatible with the cell decomposition of $\M_K^0$.

\section{Lefschetz trace formula for Lubin-Tate spaces}

 In this section $\gamma=(g,b)\in GL_n(F)\times D^\times$ is an element such that both $g$ and $b$ are regular elliptic, $gKg^{-1}=K$ and $v_p(detg)+v_p(Nrdb)=0$. (Here we use the convention that an elliptic element is always semi-simple.) Since $\gamma$ is regular elliptic, the set of $\gamma$-fixed vertices $(\B^{0})^\gamma$ is non empty, cf. \cite{SS}. Let $\wh{o}$ be a fixed choice of point in $(\B^0)^\gamma$, and $o\in \I_K$ be its image in the quotient space. One can take the above choice of $\wh{o}$ so that $\wh{o}\in \varphi^{-1}(0)^0, o\in \I_K^0$. Then $\gamma(o)=o$ by the action $\gamma:\I_K^0\ra\I_K^0$.  For any real number $\rho>0$, we consider the subset of $\I_K^0$
\[A_\rho=\{x\in\I_K^0|\;\ov{d}(o,x)\leq \rho\},\]which is a finite set for any fixed $\rho$. Moreover since $\gamma(o)=o$ and $\ov{d}$ is $\gamma$-isometric, we have $\gamma(A_\rho)=A_\rho$.
\begin{definition}For any finite set $A\subset \I_K^0$, we define two subspaces of $\M_K^0$
\[\begin{split}&V_A=\bigcup_{[T,d]\in A}\D_{[T,d],K},\\&U_A=\M_K^0-\bigcup_{[T,d]\notin A}\D_{[T,d],K}.\end{split}\]
\end{definition}

\begin{proposition}
$U_A$ is an open subspace of $\M_K^0$, while $V_A$ is a compact analytic domain, and $U_A\subset V_A$.
\end{proposition}
\begin{proof}
Since $\M_K^0-U_A=\bigcup_{[T,d]\notin A}\D_{[T,d],K}$, which is a locally finite union of closed subsets, therefore it is closed, and $U_A$ is open. $V_A$ is a finite union of compact analytic domains thus so is itself. The inclusion simply comes from the fact $\M_K^0=V_A\bigcup(\M_K^0-U_A)$.
\end{proof}

 When $\rho\ra\infty$, the finite sets $A_\rho$ exhaust $\I_K^0$. For any $\rho\geq 0$, we denote $U_\rho=U_{A_\rho}, V_\rho=V_{A_\rho}$. Since $U_\rho$ is relatively compact, we can compute the cohomology of $\M_K^0$ as
\[ H^i_c(\M_K^0\times\Cm_p,\ov{\Q}_l)=\varinjlim_{\rho}H^i_c(U_\rho\times\Cm_p,\ov{\Q}_l).\]
Moreover, for $\rho>>0$ large enough, the cohomology groups $H^i_c(U_\rho\times\Cm_p,\ov{\Q}_l)$ is constant and bijective to $H^i_c(\M_K^0\times\Cm_p,\ov{\Q}_l)$, see proposition 3.5.

  For the $\gamma$ above, we consider the action $\gamma:\M_K^0\ra\M_K^0$. Since $\gamma(A_\rho)=A_\rho$, \[\gamma(U_\rho)=U_\rho, \gamma(V_\rho)=V_\rho.\] $\gamma$ thus acts also on the cells contained in $V_\rho$: $\gamma(\D_{[T,d],K})=\D_{[Tg,bd],K}$. Passing to cohomology,
  $\gamma$ induces an automorphism
\[\gamma: H^i_c(U_\rho\times\Cm_p,\ov{\Q}_l)\ra H^i_c(U_\rho\times\Cm_p,\ov{\Q}_l).\]Consider
\[H^\ast_c(U_\rho\times\Cm_p,\ov{\Q}_l)=\sum_{i}(-1)^iH^i_c(U_\rho\times\Cm_p,\ov{\Q}_l)\] as an element in some suitable Grothendieck group, and the trace of $\gamma$
\[Tr(\gamma|H^\ast_c(U_\rho\times\Cm_p,\ov{\Q}_l))=\sum_{i}(-1)^iTr(\gamma|H^i_c(U_\rho\times\Cm_p,\ov{\Q}_l)).\]Let
$\textrm{Fix}(\gamma|\M_K^0\times\Cm_p)$ be the set of fixed points of $\gamma$ on $\M_K^0\times\Cm_p$, then each fixed point is simple since the $p$-adic period mapping is \'etale (cf. \cite{St} theorem 2.6.8).

We will use our result of cell decomposition of $\M_K^0$, to verify that the action of $\gamma$ satisfies the conditions of Mieda's theorem 3.13 \cite{Mi1}, thus deduce a Lefschetz trace formula in our case. Recall that, if $k$ is a complete non-archimedean field and $k^0$ is its ring of integers, then the category of Hausdorff strictly $k$-analytic spaces is equivalent to the category of adic spaces which are taut and locally of finite type over $spa(k,k^0)$, see \cite{Hu1} chapter 8. If $X$ is a Hausdorff strictly $k$-analytic space, we denote by $X^{ad}$ the associated adic space, which is taut and locally of finite type over $spa(k,k^0)$.

\begin{theorem}
Let the notations and assumptions be as above. There exist an open compact subgroup $K'\subset GL_n(O)$ and a real number $\rho_0$, such that for all open compact subgroup $K\subset K'$ which is normalized by $g$ and all $\rho\geq\rho_0$, we have
\[Tr(\gamma|H^\ast_c(U_\rho\times\Cm_p,\ov{\Q}_l))=\#\textrm{Fix}(\gamma|\M_K^0\times\Cm_p).\]
For $\rho$ sufficiently large, the left hand side is just $Tr(\gamma|H^\ast_c(\M_K^0\times\Cm_p,\ov{\Q}_l))$.
\end{theorem}
\begin{proof}
Since $g\in GL_n(F)$ is elliptic, we first note the following fact:
for any sufficiently small open compact subgroup $K\subset GL_n(O)$ such that $gKg^{-1}=K$, we have
\[\ov{d}(x,\gamma x)\ra\infty,\;\textrm{when}\; x\in \I_K^0,  \ov{d}(o,x)\ra\infty.\]

In fact, since $o,x\in \I_K^0$, write $o=[o'K,-s,s], x=[x'K,-t,t]$ with $o',x'\in \B(PGL_n,F)^0, s,t\in [0,n-1]\bigcap\Z$, then $\gamma(x)=[x'gK,v_p(detg)-t,v_p(Nrdb)+t]=[x'gK,-t',t']$ for some unique $t'=v_p(detg)-t+nr\in [0,n-1]\bigcap\Z$. If we denote the metric on $\B'=\B(PGL_n,F)$ by $d'(\cdot,\cdot)$ and the induced metric on $\B'/K$ by $\ov{d'}$ as before, then we just need to prove that
\[\ov{d'}(x'K, x'gK)\ra\infty,\;\textrm{when}\; x'K\in (\B')^0/K,  \ov{d'}(o'K,x'K)\ra\infty.\]
To prove this statement, we first work with $\B'$ itself by not the quotient. Since $g$ is elliptic, the fixed points set $(\B')^{g}$ is nonempty and compact. Moreover, for $K$ sufficiently small, $(\B')^g=(\B')^{g'}$ for any $g'\in gK$ (cf. the proof of lemma 12 in \cite{SS}). For $o'\in (\B')^g$ fixed, a simple triangle inequality shows that $d'(x',(\B')^g)\ra\infty$ when $d'(x',o')\ra\infty$, since $(\B')^g$ is compact. On the other hand, for any automorphism $\sigma$ of $\B'$ with $(\B')^\sigma\neq\emptyset$, there exists a constant $0<\theta\leq \pi$ which just depends on $\B'$ and $\sigma$, such that \[d'(x',\sigma x')\geq 2d'(x',(\B')^\sigma)\sin(\frac{\theta}{2}),\]
see \cite{R} proposition 2.3. In particular, $d'(x', x'g')\ra\infty$ when $d'(o',x')\ra \infty$ for any $g'\in gK$. As $K$ is compact this deduces the above statement.

For $\rho$ sufficiently large,
\[\begin{split}
\M_K^0-U_\rho&=\bigcup_{[T,d]\in \I_K^0-A_\rho}\D_{[T,d],K}\\
V_\rho-U_\rho&=\bigcup_{[T,d]\in {A_\rho-A_{\rho-c}}}F_{[T,d]},\end{split}\]where for $[T,d]\in A_\rho$, \[F_{[T,d]}=\D_{[T,d],K}\bigcap(\M_K^0-U_\rho),\] which is nonempty if and only if $[T,d]\in A_\rho-A_{\rho-c}$ by the above proposition, in which case  $F_{[T,d]}$ is a compact analytic domain in $\D_{[T,d],K}\subset V_\rho$. For $K$ sufficiently small, $\rho$ sufficiently large and $[T,d]\in \I_K^0-A_{\rho-c}$, by the lemma $\ov{d}([T,d],\gamma([T,d]))>c$, thus \[\D_{[T,d],K}\bigcap\gamma(\D_{[T,d],K})=\emptyset, \; F_{[T,d]}\bigcap\gamma(F_{[T,d]})=\emptyset \;(\textrm{for}\;[T,d]\in A_\rho-A_{\rho-c}).\]

To apply Mieda's theorem, we pass to adic spaces. We have the locally finite cell decomposition of the adic space $(\M_K^0)^{ad}$:
\[(\M_K^0)^{ad}=\bigcup_{[T,d]\in\I_K^0}\D^{ad}_{[T,d],K}\]where each cell $\D^{ad}_{[T,d],K}$ is an open quasi-compact subspace, $\D^{ad}_{[T_1,d_1],K}\bigcap\D^{ad}_{[T_2,d_2],K}\neq \emptyset \Leftrightarrow \D_{[T_1,d_1],K}\bigcap\D_{[T_2,d_2],K}\neq \emptyset$, and the action of $\gamma$ on $(\M_K^0)^{ad}$ induces an action on the cells in the same way as the case of Berkovich analytic spaces. By \cite{Hu1} 8.2, $U_\rho^{ad}$ is an open subspace of $(\M_K^0)^{ad}$, which is separated, smooth, partially proper. On the other hand, $V_\rho^{ad}=\bigcup_{[T,d]\in A_\rho}\D^{ad}_{[T,d],K}$ is a quasi-compact open subspace. Consider the closure $\ov{V_\rho^{ad}}=\bigcup_{[T,d]\in A_\rho}\ov{\D^{ad}_{[T,d],K}}$ of $V_\rho^{ad}$ in $(\M_K^0)^{ad}$, which is a proper pseudo-adic space. We know that $\ov{V_\rho^{ad}}$ (resp. $\ov{\D^{ad}_{[T,d],K}}$) is the set of all specializations of the points in $V_\rho^{ad}$ (resp. $\D^{ad}_{[T,d],K}$). Moreover $\gamma$ induce automorphisms $\gamma: \ov{V_\rho^{ad}}\ra \ov{V_\rho^{ad}}, V_\rho^{ad}\ra V_\rho^{ad}$, $U_\rho^{ad}\ra U_\rho^{ad}$. Since $V_{\rho-c}^{ad}\subset U_\rho^{ad}\subset V_\rho^{ad}$, we have $\ov{V_\rho^{ad}}-V_\rho^{ad}=\bigcup_{[T,d]\in A_\rho-A_{\rho-c}}(\ov{\D^{ad}_{[T,d],K}}-\D^{ad}_{[T,d],K})$. Note \[ \D^{ad}_{[T_1,d_1],K}\bigcap \D^{ad}_{[T_2,d_2],K}\neq\emptyset
 \Leftrightarrow \ov{\D^{ad}_{[T_1,d_1],K}}\bigcap \ov{\D^{ad}_{[T_2,d_2],K}}\neq\emptyset.\]For $[T,d]\in A_\rho-A_{\rho-c}$, let $W_{[T,d]}=\ov{\D^{ad}_{[T,d],K}}-\D^{ad}_{[T,d],K}$. By the paragraph above, for $\rho>>0$ we have $\gamma(W_{[T,d]})\bigcap W_{[T,d]}=\emptyset$. One sees the conditions of theorem 3.13 of \cite{Mi1} hold for $V_\rho^{ad}$ and its compactification $\ov{V_\rho^{ad}}$, i.e.
  \[Tr(\gamma|H^\ast_c(V^{ad}_\rho\times\Cm_p,\ov{\Q}_l))=\#\textrm{Fix}(\gamma|V^{ad}_\rho\times\Cm_p)
  =\#\textrm{Fix}(\gamma|V_\rho\times\Cm_p).\]Here and in the following $V^{ad}_\rho\times\Cm_p:=V^{ad}_\rho\times spa(\Cm_p,O_{\Cm_p})$, and similar notations for other adic spaces. By \cite{Hu2} proposition 2.6 (i) and lemma 3.4, we have
    \[Tr(\gamma|H_c^\ast(V^{ad}_\rho\times\Cm_p,\ov{\Q}_l))=Tr(\gamma|H_c^\ast(U^{ad}_\rho\times\Cm_p,\ov{\Q}_l))+Tr(\gamma|
  H_c^\ast((V^{ad}_\rho-U_\rho^{ad})\times\Cm_p,\ov{\Q}_l)).\] By the paragraph above one can see it easily by the induction argument of the proof of proposition 4.10 in \cite{Mi1} that $Tr(\gamma|H_c^\ast((V^{ad}_\rho-U^{ad}_\rho)\times\Cm_p,\ov{\Q}_l))=0$. Thus we can conclude by Huber's comparison theorem on compactly support cohomology of Berkovich spaces and adic spaces (cf. proposition 8.3.6 of \cite{Hu1}),
\[Tr(\gamma|H^\ast_c(U_\rho\times\Cm_p,\ov{\Q}_l))=Tr(\gamma|H^\ast_c(U_\rho^{ad}\times\Cm_p,\ov{\Q}_l))
=Tr(\gamma|H_c^\ast(V^{ad}_\rho\times\Cm_p,\ov{\Q}_l))=
\#\textrm{Fix}(\gamma|V_\rho\times\Cm_p).\]
But as reason above for $\rho>>0$ there is no fixed points of $\gamma$ outside $V_\rho\times\Cm_p$,
\[\#\textrm{Fix}(\gamma|V_\rho\times\Cm_p)=\#\textrm{Fix}(\gamma|\M_K^0\times\Cm_p).\]
The last statement in the theorem comes from the following proposition 3.5.

\end{proof}

\begin{remark}In fact we can use $V_\rho$ to compute the cohomology of $\M_K^0$ directly when passing to adic spaces:
\[H_c^i(\M_K^0\times\Cm_p,\ov{\Q}_l)\simeq H_c^i((\M_K^0)^{ad}\times\Cm_p,\ov{\Q}_l)=\varinjlim_{\rho}H_c^i(V^{ad}_\rho\times\Cm_p,\ov{\Q}_l),\;\forall \,i\geq0,\]
here the second equality comes from proposition 2.1 (iv) of \cite{Hu2}. We prefer to transfer back the results to Berkovich spaces, so we insist on working with the open subspaces $U_\rho$.
\end{remark}

In fact, the formal models $\wh{\M}^0_K$ are algebraizable: they are the formal completions at closed points of some Shimura varieties as in \cite{HT}, or one can find the algebraization directly as in theorem 2.3.1 in \cite{St}. So we have for all integer $i\geq 0$ \[H_c^i(\M_K^0\times\Cm_p,\ov{\Q}_l)\simeq (\varprojlim_{r}H^i_c(\M_K^0\times\Cm_p,\Z/l^r\Z))\otimes_{\Z_l}\ov{\Q}_l,
\]and similarly for the cohomology without compact support. We have the following proposition.
\begin{proposition}
Let the notations be as above. Then for $\rho>>0$ and all integer $i\geq 0$, we have bijections
\[H^i_c(\M_K^0\times\Cm_p,\ov{\Q}_l)\simeq H^i_c(V^{ad}_\rho\times\Cm_p,\ov{\Q}_l)\simeq H_c^i(U_\rho\times\Cm_p,\ov{\Q}_l).\]
\end{proposition}
\begin{proof}
This comes from the description of $V_\rho$ and Huber's theorem 2.9 in \cite{Hu3}. Recall the fundamental domain $\D\subset\M^0$ is associated to an admissible open subset $\D^{rig}\subset (\M^0)^{rig}$. On the rigid analytic space $(\M^0)^{rig}$ there is a natural coordinate system $x_1,\dots,x_{n-1}$, such that for $x=(x_1,\dots,x_{n-1})\in (\M^0)^{rig}$, the Newton polygon of multiplication by $\pi$ on the formal group law associated to $\pi$-divisible group $H_x$ is the convex envelope of the points $(q^i,v(x_i))_{0\leq i\leq n}$, where $x_0=0,x_n=1, q=\#O/\pi$, cf. \cite{F3} 1.1.5. Under this coordinate system \[\D^{rig}=\{x=(x_1,\dots,x_{n-1})\in(\M^0)^{rig}|v(x_i)\geq 1-\frac{i}{n}, i=1,\dots,n-1\},\]cf. loc. cit. 1.4. Thus after base change to $\Cm_p$ it is isomorphic to a closed ball. In \cite{F2} section 5 Fargues had described the Newton polygons of the points in a Hecke orbit. In particular at level $K=GL_n(O)$ the admissible open subsets $V^{rig}_\rho\times\Cm_p$ are locally described by closed balls. Then this is also the case for any level $K$. Now pass to adic spaces, $V^{ad}_\rho\times\Cm_p$ are quasi-compact open subsets and locally described by $\mathbb{B}_{\epsilon_\rho}=\{z\in(\M^0)^{ad}\times\Cm_p| |x_i(z)|\leq \epsilon_\rho\}$. Since $U_\rho^{ad}\times\Cm_p, (\M^0)^{ad}\times\Cm_p$ can be described as unions of ascending chains of quasi-compact open subsets locally in the above forms, by theorem 2.9 of \cite{Hu3} and proposition 8.3.6 of \cite{Hu1} one can conclude.
\end{proof}

Let $\gamma=(g,b)\in GL_n(F)\times D^\times$, with $g,b$ both regular elliptic and $gKg^{-1}=K$. For the quotient space $\M_K/\pi^\Z$ we have a nice fixed points formula by considering the $p$-adic period mapping, which is non zero if and only if $v_p(degg)+v_p(Nrdb)$ is divisible by $n$. Fix compatible Haar measures on $GL_n(F)$ and $G_{g_b}$ (see below), we can also write it as some suitable orbit integral (\cite{St} theorem 2.6.8, \cite{Mi1} proposition 3.3).
\[\begin{split}\#\textrm{Fix}(\gamma|(\M_K/\pi^\Z)\times\Cm_p)&=n\#\{h\in GL_n(F)/\pi^\Z K|h^{-1}g_bh=g^{-1}\}\\&=nVol(G_{g_b}/\pi^\Z)\int_{GL_n(F)/G_{g_b}}\frac{1_{g^{-1}K}}{Vol(K)}(z^{-1}g_bz)dz,\end{split}\]
where $g_b\in GL_n(F)$ is an element stably conjugate to $b\in D^\times$, $G_{g_b}$ is the centralizer of $g_b$ in $GL_n(F)$, $Vol(K)$ (resp. $Vol(G_{g_b}/\pi^\Z)$) is the volume of $K$ (resp. $G_{g_b}/\pi^\Z$ for the induced Haar measure), and $1_{g^{-1}K}$ is the characteristic function of $g^{-1}K$. After some representation theory arguments, the Lefschetz trace formula and the above fixed points formula will lead to a local proof of the realization of Jacquet-Langlands correspondence in the cohomology of Lubin-Tate spaces (cf. section 4 of \cite{St} and \cite{Mi2}).

\end{document}